\documentclass[12pt]{amsart}
\usepackage{amsmath,amsthm,amsfonts,amssymb,latexsym}

\usepackage{hyperref}

\usepackage{enumerate}
\usepackage[shortlabels]{enumitem}

\begin{document}

\theoremstyle{plain}

\newtheorem{thm}{Theorem}[section]

\newtheorem{lem}[thm]{Lemma}
\newtheorem{Problem B}[thm]{Problem B}

\newtheorem{pro}[thm]{Proposition}
\newtheorem{cor}[thm]{Corollary}
\newtheorem{que}[thm]{Question}
\newtheorem{rem}[thm]{Remark}
\newtheorem{defi}[thm]{Definition}

\newtheorem*{thmA}{Theorem A}
\newtheorem*{thmB}{Theorem B}
\newtheorem*{corB}{Corollary B}
\newtheorem*{thmC}{Theorem  C}
\newtheorem*{thmD}{Theorem D}
\newtheorem*{thmE}{Theorem E}

\newtheorem*{thmAcl}{Main Theorem$^{*}$}
\newtheorem*{thmBcl}{Theorem B$^{*}$}

\newcommand{\Maxn}{\operatorname{Max_{\textbf{N}}}}
\newcommand{\Syl}{\operatorname{Syl}}
\newcommand{\dl}{\operatorname{dl}}
\newcommand{\Con}{\operatorname{Con}}
\newcommand{\cl}{\operatorname{cl}}
\newcommand{\Stab}{\operatorname{Stab}}
\newcommand{\Aut}{\operatorname{Aut}}
\newcommand{\Sym}{\operatorname{Aut}}
\newcommand{\Ker}{\operatorname{Ker}}
\newcommand{\fl}{\operatorname{fl}}
\newcommand{\Irr}{\operatorname{Irr}}
\newcommand{\SL}{\operatorname{SL}}
\newcommand{\FF}{\mathbb{F}}
\newcommand{\NN}{\mathbb{N}}
\newcommand{\N}{\mathbf{N}}
\newcommand{\bfC}{\mathbf{C}}
\newcommand{\bfO}{\mathbf{O}}
\newcommand{\bfF}{\mathbf{F}}
\newcommand{\OO}{\mathbf{O}}

\renewcommand{\labelenumi}{\upshape (\roman{enumi})}

\newcommand{\PSL}{\operatorname{PSL}}
\newcommand{\PSU}{\operatorname{PSU}}
\newcommand{\alt}{\operatorname{Alt}}
\newcommand{\GL}{\operatorname{GL}}
\newcommand{\SO}{\operatorname{SO}}
\newcommand{\SU}{\operatorname{SU}}

\providecommand{\V}{\mathrm{V}}
\providecommand{\E}{\mathrm{E}}
\providecommand{\ir}{\mathrm{Irr_{rv}}}
\providecommand{\Irrr}{\mathrm{Irr_{rv}}}
\providecommand{\re}{\mathrm{Re}}

\numberwithin{equation}{section}
\def\irrp#1{{\rm Irr}_{p'}(#1)}

\def\ibrrp#1{{\rm IBr}_{\Bbb R, p'}(#1)}
\def\Z{{\mathbb Z}}
\def\C{{\mathbb C}}
\def\Q{{\mathbb Q}}
\def\irr#1{{\rm Irr}(#1)}
\def\aut#1{{\rm Aut}}
\def\ibr#1{{\rm IBr}(#1)}
\def\irrp#1{{\rm Irr}_{p^\prime}(#1)}
\def\irrq#1{{\rm Irr}_{q^\prime}(#1)}
\def \c#1{{\cal #1}}
\def\cent#1#2{{\bf C}_{#1}(#2)}
\def\syl#1#2{{\rm Syl}_#1(#2)}
\def\nor{\triangleleft\,}
\def\oh#1#2{{\bf O}_{#1}(#2)}
\def\Oh#1#2{{\bf O}^{#1}(#2)}
\def\zent#1{{\bf Z}(#1)}
\def\det#1{{\rm det}(#1)}
\def\ker#1{{\rm ker}(#1)}
\def\norm#1#2{{\bf N}_{#1}(#2)}
\def\alt#1{{\rm Alt}(#1)}
\def\iitem#1{\goodbreak\par\noindent{\bf #1}}
   \def \mod#1{\, {\rm mod} \, #1 \, }
\def\sbs{\subseteq}

\def\gc{{\bf GC}}
\def\ngc{{non-{\bf GC}}}
\def\ngcs{{non-{\bf GC}$^*$}}
\newcommand{\notd}{{\!\not{|}}}
\newcommand{\Out}{{\mathrm {Out}}}
\newcommand{\Mult}{{\mathrm {Mult}}}
\newcommand{\Inn}{{\mathrm {Inn}}}
\newcommand{\IBR}{{\mathrm {IBr}}}
\newcommand{\IBRL}{{\mathrm {IBr}}_{\ell}}
\newcommand{\IBRP}{{\mathrm {IBr}}_{p}}
\newcommand{\ord}{{\mathrm {ord}}}
\def\id{\mathop{\mathrm{ id}}\nolimits}
\renewcommand{\Im}{{\mathrm {Im}}}
\newcommand{\Ind}{{\mathrm {Ind}}}
\newcommand{\diag}{{\mathrm {diag}}}
\newcommand{\soc}{{\mathrm {soc}}}
\newcommand{\End}{{\mathrm {End}}}
\newcommand{\sol}{{\mathrm {sol}}}
\newcommand{\Hom}{{\mathrm {Hom}}}
\newcommand{\Mor}{{\mathrm {Mor}}}
\newcommand{\Mat}{{\mathrm {Mat}}}
\def\rank{\mathop{\mathrm{ rank}}\nolimits}
\newcommand{\Tr}{{\mathrm {Tr}}}
\newcommand{\tr}{{\mathrm {tr}}}
\newcommand{\Gal}{{\it Gal}}
\newcommand{\Spec}{{\mathrm {Spec}}}
\newcommand{\ad}{{\mathrm {ad}}}

\newcommand{\Char}{{\mathrm {char}}}
\newcommand{\pr}{{\mathrm {pr}}}
\newcommand{\rad}{{\mathrm {rad}}}
\newcommand{\abel}{{\mathrm {abel}}}
\newcommand{\codim}{{\mathrm {codim}}}
\newcommand{\ind}{{\mathrm {ind}}}
\newcommand{\Res}{{\mathrm {Res}}}
\newcommand{\Lie}{{\mathrm {Lie}}}
\newcommand{\Ext}{{\mathrm {Ext}}}
\newcommand{\Alt}{{\mathrm {Alt}}}
\newcommand{\AAA}{{\sf A}}
\newcommand{\SSS}{{\sf S}}
\newcommand{\CC}{{\mathbb C}}
\newcommand{\CB}{{\mathbf C}}
\newcommand{\RR}{{\mathbb R}}
\newcommand{\QQ}{{\mathbb Q}}
\newcommand{\ZZ}{{\mathbb Z}}
\newcommand{\bfN}{{\mathbf N}}
\newcommand{\bfZ}{{\mathbf Z}}
\newcommand{\EE}{{\mathbb E}}
\newcommand{\PP}{{\mathbb P}}
\newcommand{\cG}{{\mathcal G}}
\newcommand{\cH}{{\mathcal H}}
\newcommand{\cQ}{{\mathcal Q}}
\newcommand{\GA}{{\mathfrak G}}
\newcommand{\cT}{{\mathcal T}}
\newcommand{\cS}{{\mathcal S}}
\newcommand{\cR}{{\mathcal R}}
\newcommand{\GCD}{\GC^{*}}
\newcommand{\TCD}{\TC^{*}}
\newcommand{\FD}{F^{*}}
\newcommand{\GD}{G^{*}}
\newcommand{\HD}{H^{*}}
\newcommand{\GCF}{\GC^{F}}
\newcommand{\TCF}{\TC^{F}}
\newcommand{\PCF}{\PC^{F}}
\newcommand{\GCDF}{(\GC^{*})^{F^{*}}}
\newcommand{\RGTT}{R^{\GC}_{\TC}(\theta)}
\newcommand{\RGTA}{R^{\GC}_{\TC}(1)}
\newcommand{\Om}{\Omega}
\newcommand{\eps}{\epsilon}
\newcommand{\al}{\alpha}
\newcommand{\chis}{\chi_{s}}
\newcommand{\sigmad}{\sigma^{*}}
\newcommand{\PA}{\boldsymbol{\alpha}}
\newcommand{\gam}{\gamma}
\newcommand{\lam}{\lambda}
\newcommand{\la}{\langle}
\newcommand{\ra}{\rangle}
\newcommand{\hs}{\hat{s}}
\newcommand{\htt}{\hat{t}}
\newcommand{\tn}{\hspace{0.5mm}^{t}\hspace*{-0.2mm}}
\newcommand{\ta}{\hspace{0.5mm}^{2}\hspace*{-0.2mm}}
\newcommand{\tb}{\hspace{0.5mm}^{3}\hspace*{-0.2mm}}
\def\skipa{\vspace{-1.5mm} & \vspace{-1.5mm} & \vspace{-1.5mm}\\}
\newcommand{\tw}[1]{{}^#1\!}
\renewcommand{\mod}{\bmod \,}

\marginparsep-0.5cm

\renewcommand{\thefootnote}{\fnsymbol{footnote}}
\footnotesep6.5pt

\newcommand{\sym}{{\mathrm {Sym}}}
\newcommand{\Sp}{{\mathrm {Sp}}}
\newcommand{\PSp}{{\mathrm {PSp}}}

\dedicatory{Dedicated to the memory of our good friend Jan Saxl} 

\title{Real Constituents of  Permutation Characters}

\author{Robert Guralnick}
\address{Department of Mathematics, University of Southern California, 3620 S. Vermont Ave., Los Angeles, CA 90089, USA}
\email{guralnic@usc.edu}
\author{Gabriel Navarro}
\address{Departament de Matem\`atiques, Universitat de Val\`encia, 46100 Burjassot,
Val\`encia, Spain}
\email{gabriel@uv.es}
  
\thanks{The first  author
gratefully acknowledges the support of the NSF grant DMS-1901595 and a Simons Foundation
Fellowship 609771.    The research of the second author is supported by Ministerio de Ciencia
e Innovaci\'on
PID2019-103854GB-I00 and FEDER funds.
Both authors would like to thank Thomas Breuer for many helpful calculations,
and to Gunter Malle for  a careful reading of the manuscript.}



\begin{abstract}
We prove a broad generalization of a theorem of W. Burnside  about the existence of real characters of finite
groups to permutation characters.  If
$G$ is a finite group, under the necessary hypothesis 
of $\Oh{2'} G=G$,  
we can also give some control on 
the parity of  multiplicities of the constituents of permutation characters (a result that needs
the Classification of Finite Simple Groups).
Along the way, we   give a new characterization of the 2-closed
finite groups using odd-order real elements of the group. 
All this can be seen as a  contribution to Brauer's Problem 11 which asks how much
information about subgroups of a finite group can be determined by the character table.  
 \end{abstract}

\maketitle

\section{Introduction}

Recall that W.  Burnside proved that a finite group of odd order has no nontrivial irreducible
real valued characters.
This paper began with an elementary generalization of this theorem of Burnside, which we have not found in the literature.   

Before stating the result, we recall the trichotomy for    irreducible complex characters of a finite group $G$.
If $\chi$ is an irreducible complex character of $G$, then the Frobenius Schur indicator $\nu_2(\chi) =
|G|^{-1}\sum_{g \in G}  \chi(g^2)$  has possible values $0, \pm 1$  \cite[Theorem 4.5]{Is}.   We say that $\chi$ is of $+$ type if
$\nu_2(\chi)=1$.   This happens if there is
a representation affording $\chi$  which is defined over $\RR$, or equivalently, if
there is a $G$-invariant nonzero quadratic form on the corresponding module.  
We have that
$\nu_2(\chi)=-1$ if 
$\chi$ is real valued but the representation cannot defined over $\RR$. If $\chi$ is not
real valued, then $\nu_2(\chi)=0$, or equivalently,   the representation is not self dual.    
(See \cite{Is} for a more detailed discussion.)   

  \begin{thmA}
   Let $G$ be a finite group, and let $H$ be a subgroup of $G$. Then the permutation character
    $(1_H)^G$ has
a unique real-valued $+$ type irreducible constituent if and only if
 $[G:{\rm core}_G(H)]$ is odd. 
   \end{thmA}

   \begin{proof}
   Write $K={\rm core}_G(H)$. If $G/K$ has odd order, then the trivial character $1_G$ is the only real-valued
   irreducible constituent of $(1_H)^G$, by Burnside's Theorem.
   Conversely, if $|G/K|$ is not odd, then let $xK$ be an involution of $G/K$. 
   Then $|G:H|=(1_H)^G(x^2)>(1_H)^G(x)$, since $\ker{(1_H)^G}=K$. 
   Therefore, using this and the fact that $(1_H)^G$ is a permutation character, we have that
   $${1 \over |G|} \sum_{g \in G} (1_H)^G(g^2) >
   {1 \over |G|} \sum_{g \in G} (1_H)^G(g)=[(1_H)^G, 1_G]=1 \, .$$
   Then
   $$1<{1 \over |G|} \sum_{g \in G} (1_H)^G(g^2)=\sum_{\chi \in \irr G} [(1_H)^G, \chi] \nu_2(\chi). $$
Since    $\nu_2(\chi)$ is $1, -1$ or $0$,  
    the proof easily follows.
   \end{proof}

    \medskip
  In most  cases that we have observed,  it turns out that there is
    a non-trivial real-valued {\sl odd multiplicity} irreducible
   constituent $\chi$ in $(1_H)^G$ 
   as long as $[G:{\rm core}_G(H)]$ is not odd.    
   (As we shall point out in Lemma \ref{bob} below, this implies
   that $\chi$ has $+$ type.) Multiplicities of constituents of irreducible characters in  permutation characters correspond
   to the dimensions of the irreducible representations of the Hecke algebra ${\rm End}((1_H)^G)$, so they
   are of interest.  
   
   \medskip
   
    We now give an example to show that, unfortunately, it is not always the case that there is a nontrivial
    real valued irreducible constituent of odd multiplicity  in $(1_H)^G$ when $[G: {\rm core}_G(H)]$
    is even.   Suppose that $p \equiv 3$ mod 4 is a prime, and let $q=p^a$, where
    $a>1$ is odd. Let $G={\rm AGL}(1,q)$, the semidirect product
    of the Galois field $F={\rm GF}(q)$ with
    the multiplicative group ${\rm GF}(q)^\times$. If $H \le G$ has order $2p$, then
    $G/FH$ is a cyclic group of odd order $q-1\over 2$, and 
    $$(1_H)^G={q^{a-1} -1 \over 2} \theta +\sum_{\lambda \in \irr{G/FH}} \lambda    \, ,$$
    where $\theta \in \irr G$ is the unique irreducible character of order $q-1$.
    Also, notice that  $q^{a-1} -1\over 2$ is even and $\theta$ is rational-valued. 
   \medskip

   Our main result in this paper is that this does not happen under the 
   additional hypothesis that $\Oh{2'} G=G$. (Recall that $\Oh{2'}G$
   is the smallest normal subgroup $N$ of $G$ such that $G/N$ has odd order.)
   For the proof, we require the Classification of Finite Simple Groups.
   
\begin{thmB}
Let $G$ be a finite group and let $H$ be a proper subgroup of
$G$.   If $\Oh{2'}G=G$, then $(1_H)^G$ has a
non-trivial real-valued irreducible constituent 
of odd multiplicity.
\end{thmB}

    There is a great deal of literature on permutation characters, much inspired by Richard Brauer's Problem 11 of his celebrated
    list:    {\sl Given the character table of a group $G$, how much information about the existence of subgroups can be obtained?} Notice that Theorem B adds, in principle,  a new condition to check for a given character of a finite group if it is a permutation character or not. (See  \cite[Theorem 5.18]{Is}).
    
   \medskip
 
   As we shall see, a key  idea in the proof of Theorem B is that if $1_G$
   is the only real-valued irreducible
   constituent of $(1_H)^G$ with odd multiplicity, then $[G:H]$
   is odd, and  every real conjugacy
   class of $G$ meets $H$. (See Lemma \ref{intersect}.)
    If $H$ is, say,  a non-normal Sylow 2-subgroup
    of $G$, then an elementary argument
    (using Baer's theorem and involutions, 
    see Proposition 6.4
 of \cite{DNT}) produces a non-trivial
    odd order real element, hence proving Theorem B in this case.
      In general, however, there are
     examples of proper subgroups $H$ meeting all the real classes of 
    groups $G$. 
    And in order to prove Theorem B, we shall need to classify, essentially,
  all the cases whenever $H$ is a maximal subgroup of  odd index in a simple group $G$
  and $H$ intersects every real conjugacy class. 
  We believe that this result  might have some interest on its own.   Indeed,
  we classify all primitive permutation groups in which every real element has a fixed
  point.  We only state the result for $n$ odd since that is what we use.   (See 
  Theorem \ref{primitiverealeven} for the corresponding result for $n$ even.)
  
  \begin{thmC}  Let  $G$ be a primitive permutation group acting on
a set $X$ of odd cardinality.    Assume that every real element of the socle of $G$ fixes a point
of $X$.   Let $M$ be a point stabilizer.  Then  $G$ has a unique minimal 
normal subgroup $A$ and either $G$ has odd order or 
$A = L_1 \times \ldots \times L_t$   with $L_i \cong L$ a nonabelian simple
group. Also,  $M \cap A = D_1 \times \ldots \times D_t$,  $D_i$
is maximal in $L_i$ and intersects every real conjugacy class of $G$ which
intersects $L_i$.   Furthermore 
$(1_M)^G$  has a nontrivial orthogonal constituent of multiplicity $1$.
Moreover either $G/A$ has odd order or $L \cong \PSL_3(4)$ and 
$G/A_1$ has odd order where $A_1/A$ is  an elementary abelian $2$-group
and $A_1$ normalizes each $L_i$.  
Indeed,  $L_i \cong M_{22}, M_{23}$ or $\PSL_n(q)$ with $n$ odd
and $D_i$ is as given in Theorem \ref{simplereal}. 
 \end{thmC}

  \medskip
  
  In the special case where
   $H$ is a proper 2-Sylow normalizer of $G$, a case with particular interest,
   we
   shall prove (again using the Classification)
   that not only $H$ does not meet some real
   class of $G$,  but something stronger.       
   \medskip
   
   \begin{thmD}
Let $G$ be a finite group, and let $P \in \syl 2 G$. Then the following
are equivalent.
\begin{enumerate}
\item  $P$ is normal in $G$.
\item  There are no nontrivial real elements in $G$ of odd order.
\item  
The 2-Brauer character $({1_{\norm GP}}^0)^G$
has no nontrivial  
real-valued 
$2$-Brauer irreducible constituent with odd multiplicity. 
\item  Every real element of odd order of $G$ normalizes a $G$-conjugate of $P$.
\end{enumerate}
\end{thmD}

We remark that the equivalence of the first three conditions
is elementary, while the fourth is proved 
 using the Classification.   
 
 %%%%% can we make have + type   %%%%

  \medskip
  
 Theorem D(iii) naturally suggests a question: is Theorem  B true for Brauer characters? And the answer is
 ``no." For instance $G={\rm PSL}_3(2)$ has  $H={\sf S}_4$  as a subgroup
 and the 2-Brauer character $({1_H}^0)^G$ decomposes as the trivial 2-Brauer
 character of $G$ plus two non-real-valued irreducible 2-Brauer characters
 of degree 3.   Indeed, Theorem B fails for $p$-Brauer characters for any $p$.
 Let $G=\PSL_n(p)$ with $n$ odd.  By \cite[1.8]{ZS},   the   composition factors in characteristic $p$
 of $(1_H)^G$, with $H$ the stabilizer of a $1$-space, are the trivial module and 
 $(p-1)\omega_i$ where the $\omega_i$ are the fundamental dominant weights for $\SL_n(p)$.
 In particular, since $n$ is odd, none of the nontrivial composition factors are self dual and
 so none of the nontrivial  irreducible $p$-Brauer characters are real.  
 \medskip

On the other hand, in Section \ref{sec:2brauer} below, we shall prove that analogs of
Theorems A and B do  hold for subgroups of odd index and  for 2-Brauer characters if certain composition factors
are not involved.
 
 \medskip

It is often the case that theorems on real-valued characters can be improved to
rational-valued characters. This is clearly not the case with the results in this paper. For instance,
if $G={\sf D}_{10}$ is the dihedral group
 and $P \in \syl 2 G$, then $(1_P)^G$ has a unique rational-valued irreducible constituent
 with odd multiplicity. 
 
\medskip

\section{Primitive Groups and Real Elements} \label{sec:primitive} 

In this section, we classify the primitive permutation groups such that all real elements
have fixed points.   We apply the result in the case that the degree is odd but the methods to deal with
the general case are only a bit more onerous than the odd case and there are very few  examples in the even case.        

The critical case is when the group is almost simple.       The existence of the Mathieu group examples shows that
it is unlikely that there is a proof without the classification of finite simple groups.  

Before we start the proofs of the simple case, we note the following elementary results.

\begin{lem} \label{l:realityinnormal}  Let $N$ be a normal subgroup of $G$ and let $x \in N$.
\begin{enumerate}
\item If $x$ is real in $N$, then so is $x^g$ for any $g \in G$.
\item If $[G:N\cent Gx]=m$ is odd and $x$ is real in $G$, then 
$x$ is real in $N$.
\end{enumerate} 
\end{lem}

\begin{proof} The first statement is obvious.    Consider (ii).  Then
$x^G$ is the union of $m$ distinct conjugacy classes of $N$.
Note that inversion permutes these $m$ conjugacy classes and since
$m$ is odd, it must fix at least one of them.  Thus, some element of
$x^G$ is real in $N$ and so by (i), $x$ is real in $N$. 
\end{proof}

In particular, this gives the following result:

\begin{cor} \label{c:unipotent}  Let $G=\SL_n(q)$ and let 
$g \in G$ be unipotent.  Assume that either $q$ is even or
that $g$ has a Jordan block of odd size (which is always the
case when $n$ is odd).   Then $g$ is real in $G$.
\end{cor}

\begin{proof}   By Jordan canonical forms,  $g$
is real in $H = \GL_n(q)$.   Our assumptions imply that
$[H:G\cent Hg]$ is odd and so the result follows by the previous lemma.
\end{proof} 

Note that if $q \equiv 3 \mod 4$, then the transvections in $\SL_2(q)$ are not real.

\medskip

We shall use in several parts of this paper another elementary
result (for instance, Lemma 3.1(d) of \cite{DMN}): if
a 2-group $S$ acts non-trivially on a group of odd order $T$,
then there exists $1 \ne t \in T$ and $s \in S$ such that $t^s=t^{-1}$.

We state the results for odd and even degree separately.

\begin{thm} \label{simplereal}  Let  $L$ be a finite nonabelian simple group.   Let  $M$ be
a maximal  subgroup of $L$ such that $[L:M]$ is odd and $M$ intersects every 
  real conjugacy class of $L$.   
   Then one of the following holds:
\begin{enumerate}
\item  $L =M_{22}$,  $M = 2^4:\Alt_6$,  $[L:M]=77$ and $(1_M)^L = 1a+21a+55a$; 
\item $L=M_{22}$,   $M=2^4:\sym_5$,  $[L:M]=231$ and $(1_M)^L =1a+21a+55a+154a$; 
\item  $L = M_{23}$, $M=M_{22}$,  $[L:M]=23$ and $(1_M)^L= 1a+22a$; 
\item  $L = M_{23}$, $M=\PSL_3(4).2_2$, $[L:M]=253$ and $(1_M)^L = 1a+22a+230a$; 
\item  $L = M_{23}$, $M=2^4:\Alt_7$, $[L:M]=253$ and   $(1_M)^L= 1a+22a+230a$; 
\item $L = M_{23}$, $M=2^4:(3 \times \Alt_5).2$, $[L:M]=1771$ and   $(1_M)^L= 1a+22a+230aa+253a+1035a$.
\item  $L=\PSL_n(q)$ with $n \ge 3$ and odd,   $M$ is the stabilizer of 
$1$-space or hyperplane, $[L:M]={q^n-1 \over q-1}$ and $L$ acts $2$-transitively on the cosets
of $M$.   
\end{enumerate} 
\end{thm}  

\begin{thm} \label{simplerealeven}  Let  $L$ be a finite nonabelian simple group.   Let  $M$ be
a maximal  subgroup of $L$ such that $[L:M]$ is  even and $M$ intersects every 
real conjugacy class of $L$.  
   Then one of the following holds:
\begin{enumerate}
\item  $L=M_{11}$,  $M=\sym_5$, $[L:M] = 66$ and $(1_M)^L =  1a+10a+11a+44a$;
\item  $L=M_{23}$,   $M =  M_{11}$, $[L:M]=1288$ and $(1_M)^L =  1a+22a+230a+1035a$;
\item   $L =M_{24}$,    $M = M_{12}.2$,  $[L:M] = 1288$ and $(1_M)^L =  1a+252a+1035a$.
\end{enumerate}
\end{thm} 

We prove the two results together.  

\begin{proof}   If $L$ is a sporadic simple group, then this is quickly computed in GAP.
 (We thank
Thomas Breuer for doing the computation \cite{TB}.) In the statement of Theorem
\ref{simplereal},
$230aa$, for instance, means that the unique 
irreducible character of degree 230 of $M_{23}$
appears with multiplicity 2.

Suppose that $L = \Alt_n$ with $n \ge 5$.   First suppose that $n$ is odd and $n \ge  9$.
Note that the following  $\sym_n$-conjugacy classes   consist of real elements in $\Alt_n$:  $n-2$ cycles,
elements with exactly three orbits of size $3,3$ and $n-6$ and $3$-cycles.  Since the first
two classes intersect $M$,  it follows that $M$ is primitive and since it contains a $3$-cycle,
$M=L$, a contradiction.

Next suppose that $n$ is even and $n \ge 10$.   Then consider the following  $\sym_n$ real classes in $\Alt_n$:
$n-3$ cycles,   elements with two cycles of size $4$ and
$n-4$, elements with orbits of 
size $1, 3, 3$ and $n-7$  and  $3$-cycles.  Again it follows that $D$ is primitive and contains $3$-cycles,  a contradiction.

If $n < 9$, one can compute using GAP (or check by hand) that there are no possibilities.

So $L$ is a finite simple group of Lie type in characteristic $p$.   Let $X$ be the corresponding
algebraic group with $X$ simply connected.   We will assume that $L = X^F/Z(X^F)$ where $F$
is the corresponding Lang-Steinberg endomorphism.  The very few cases where this fails can
be checked in GAP (and many arise in different contexts -- e.g $\Sp_4(2) \cong \sym_6$).

First consider the case that $L=\PSL_2(q)$ with $q \ge 11$ and $q$ odd (if $q=5,7$ or $9$,
these groups arise in other forms).   Then  elements of order $(q \pm 1)/2$ are real and 
no proper subgroup contains elements of both orders.    If $L = \PSL_2(q)$ with $q \ge 4$
and even, then all elements are real. 

Next suppose that $-1$ is in the Weyl group of $L \ne \PSL_2(q)$.     Then all semisimple elements are real.
We claim that there exist nontrivial real unipotent elements.  If the characteristic of $L$ is even, then any
element of order $2$ is real and unipotent.     Suppose that the characteristic is not $2$.
It suffices to prove this for groups of rank at most $2$.   In  $\PSL_3(q)$ or $\PSU_3(q)$, transvections are real.
Since $\Sp_4(q)$ contains $\GL_2(q)$, the claim holds for $\Sp_4$.   Since $G_2(q)$ contains $\SL_3(q)$,
the claim holds there as well.   The only remaining case is ${^2}G_2(q)'$ with $q$ an odd power of $3$.  
Then $\PSL_2(8)$ is a subgroup and contains real elements of order $3$.   

Thus every prime dividing the order of $L$ also divides the order
of $M$.   By \cite[10.7]{LPS}, this can only occur in the following cases:

\begin{enumerate}
\item $L \cong \Omega_{2n+1}(q)$, $n \ge 2$, $n$ even  and $M$
is the stabilizer of a nondegenerate hyperplane of $-$ type 
(when $q$ is even,  $L=\Sp_{2n}(q)$);
\item  $L \cong \Omega^+_{2n}(q)$, $n \ge 4$, even and $M$ is the stabilizer of a
nondegenerate hyperplane;
\item  $L \cong \PSp_4(7)$ and $M=\Alt_7$;     %%%% involutions
%%% \item  $L \cong \Sp_4(8)$ and $M={^2}B_2(8)$;    %%%% involutions
\item  $L \cong \Sp_6(2)$ and $M = \sym_8$;       %%% class of elements of order 3, elt of order 9
\item  $L \cong \Omega^+_8(2)$ and $M$ a maximal end node parabolic or $M=\Alt_9$;   %%% elt of order 15
\item  $L \cong G_2(3)$ and $M=\PSL_2(13)$; or       %%% element of order 9
\item  $L \cong {^2}F_4(2)'$ and $M = \PSL_2(25)$.     %%%% involuiion
\end{enumerate}

In all but the first two cases, the result follows easily by inspection of the character tables.
Suppose that $L \cong \Omega_{2n+1}(q)$, $n \ge 2$ with $n$ even and $M$
the stabilizer of a nondegenerate hyperplane of $-$ type.    Let $x$ be a real element stabilizing
two complementary totally singular subspaces each of dimension $n$ which generate a hyperplane
of $+$ type.   Then $x$ is real and does not fix a hyperplane of $-$ type and so is not conjugate
to an element of  $M$.  

Suppose that $L \cong \Omega^+_{2n}(q)$, $n \ge 4$, even and $M$ is the stabilizer of a
nondegenerate hyperplane.   Let $x$ be a real element stabilizing
two complementary totally singular subspaces each of dimension $n$.   Then $x$ does not
preserve any hyperplane and in particular is not conjugate to an element of $M$. 

So we may assume that $-1$ is not in the Weyl group of $L$. 

First consider the case that $L = \Omega^{\epsilon}_{2n}$, $n \ge 5$ and odd.   Consider the following
real conjugacy classes.   The first is an element of order $(q^{n-1}+1)$ acting irreducibly on a
nondegenerate subspace of codimension $2$ (of $-$ type) and trivial on its perpendicular complement. 
The second is an element of order $q^{n-1}-1$ acting irreducibly on each of a pair of totally 
singular spaces of dimension $n-1$ which generate a nondegenerate subspace of dimension $2n-2$
of $+$ type and again acting trivially on the $2$ dimensional orthogonal complement.   

Then $M$ must act either absolutely irreducibly or is irreducible on a nondegenerate hyperplane. 
Moreover, $M$ is not a subfield group either.

Consider a unipotent element  $x$ with two Jordan blocks of size $n$.   Note that $x$ is real
in the algebraic group (since it is contained in $\GL_n$).  Since $n$ is odd, this class is
invariant under the graph automorphism and any Frobenius automorphism.   By
\cite[Theorem 3.1]{LS}, the centralizer of $x$ is connected (in $\SO_{2n}$) and so
we can choose $x$ real in $L$ (if $q$ is even, this is clear -- if $q$ is odd, we can also
see this element in $\SO_n(q) \times \SO_n(q)$ and the inverting element can be chosen
of spinor norm $1$).   Since $x$ is conjugate to an element of $M$,  $M$ cannot preserve
a nondegenerate hyperplane and so must act irreducibly.      

 Note that root elements 
are real and  so by either \cite[Theorem 1]{kantor} or \cite[Theorem 7.1]{gursaxl} there
are no possibilities for  irreducible $M$ intersecting all the conjugacy classes above.

Next consider $L=E^{\epsilon}_6(q)$.    Then long root elements are real  as well as elements in 
the unipotent class  $E_6(a_1)$  \cite[Chaper 22]{lieseitz}   The latter class rules out $M=F_4$. 
  Also every semisimple conjugacy class which
intersects $F_4(q)$ is real.  In particular,  $M$ contains elements of order $q^4 \pm 1$ and
$q^4-q^2 + 1$.  In particular, this implies that $M$ is not parabolic.  

We claim that there exists a conjugacy class of real regular unipotent elements.   If $p \ge 3$,
this follows from the fact that a regular unipotent element is real in the algebraic group and
the number of connected components of the centralizer is odd.   If $p=2$, it is still true
that there is a class of real regular unipotent elements by inspection of the character table
of $F_4(2)$ (and noting that regular unipotent elements in $F_4$ are also regular unipotent
in $E_6$).    A list of all possible maximal subgroups containing a regular unipotent element
in an exceptional group is given in \cite{gurmalle} and none of them contain all the real elements
described above.

Next suppose that $L=\PSL_n(q)$ with $n \ge 3$.    
%% Note that applying an outer automorphism does not affect
%%semisimple classes (taking inverses  or applying Frobenius preserves the conjugacy classes of cyclic
%%subgroups and diagonal automorphisms preserve all semisimple classes).    Thus, we can work with
%%$L$-classes of real semisimple elements.   
We first show that if $M$ is reducible, then $n$ is odd and $M$ must be the stabilizer of a $1$-space
or hyperplane.  

Suppose that $n$ is odd.   Then there exists a real element  $x \in L$  acting irreducibly on a hyperplane
(take an irreducible element in  $\Sp_{n-1}(q)  \le \SL_{n-1}(q)$).   
Thus, if $M$ is reducible,  the only possibility is that $M$ is the stabilizer
of a hyperplane or $1$-space.   Note that since $\SL_n(q)$ has odd
center, real elements lift to real elements in $\SL_n(q)$.   Any real element of 
$\SL_n(q)$ must have $1$ as eigenvalue and so fixes a $1$-space and
a hyperplane and so this example is allowed in the conclusion.
Suppose that $n \ge 4$ is even.   Then there is a real element $x$ that acts irreducibly and so
is not in any parabolic.   

So we may assume that $M$ is irreducible.   Note that long root elements are real.  If $q$ is odd,
it follows by \cite{kantor} that $M$ would have to be a symplectic group, a linear group or a unitary
group.     Since  $M$ is irreducible, $M$ is either a symplectic, unitary or linear group
over some subfield \cite{kantor}.   If $n$ is even, there are irreducible real elements  not defined
over any subfield (and clearly not in any unitary group) and then
the only possibility for $M$ is $\PSp_n(q)$.    If we take $x$ to be unipotent with single nontrivial
Jordan block which is of size $3$, then $x$ is real but is not contained in a symplectic group.
If $n$ is odd and $M$ is irreducible, then there are real elements of order $q^{(n-1)/2}+1$
acting irreducibly on a hyperplane and again this rules out all the possibilities for $M$.

Essentially the same proof works when $q$ is even aside from a very few small cases.  
Note that all unipotent elements are real by Corollary \ref{c:unipotent}. 
Now argue as above using either \cite{kantor} or for $d \ge 6$ \cite[Theorem 7.1]{gursaxl} to conclude
there are no such groups.   If $d < 6$, a straightforward inspection of the maximal subgroups yields
the result \cite{lowdim}.  

Finally 
suppose that $L=\PSU_n(q)$ with $n \ge 3$.    If $n = 3$, then an element of order $q+1$ with eigenvalues
$1, a, a^{-1}=a^q$ is real and is contained in no parabolic.   If $n=4$, then similarly as long as $q > 2$, we can
find a real element of order $q+1$ with distinct eigenvalues not contained in any parabolic subgroup.  Note that 
$L=\PSU_4(2) \cong \PSp_4(3) $, a case already handled.   Since transvections are real, by inspection there
are no proper subgroups intersecting each real class of $L$. 

So assume that $n \ge 5$.   First suppose that $n$ is odd.   The there exists a conjugacy of real regular
unipotent elements (since the centralizer has an odd number of components in the algebraic group
and the element is real in the algebraic group).   Thus,  $M$ cannot preserve a nondegenerate space.
There are real elements in $L$ that have irreducible submodules of dimensions $(n-1)/2, n(n-1)/2$ and $1$
(all nonisomorphic) with the first two totally singular.  Thus, the only possible parabolic that $M$ could be
is the stabilizer of totally singular $(n-1)/2$-dimensional space.   We can choose $y \in L$ real that
has irreducible non-isomorphic modules of dimensions $(n-3)/2, (n-3)/2$ and $3$ one-dimensional modules
(all nonsingular).    Thus element does not stabilize a totally singular $(n-1)/2$ space and so $M$
must be irreducible.  Since transvections are real, $M$ is absolutely irreducible.   Applying \cite{kantor} or \cite{gursaxl}  as in
the linear case, we see that there are no possibilities for $M$.  

Next suppose that $n \ge 6$ is even.  Note that unipotent elements with two Jordan blocks of size $n/2$ are real
(since they are in $\GL_{n/2}(q^2)$).   A unipotent element with a single Jordan block of size $n-1$ is also
real.  Thus, $M$ cannot preserve a nondegenerate space.   There exists $x \in L$ real with irreducible
submodules of dimensions $n/2$ and $n/2$ and also one with irreducible submodules of dimensions
$n/2 -1, n/2 -1, 1$ and $1$ where the $1$-dimensional submodules are nondegenerate.  This rules out
the possibility that $M$ is parabolic.   As in the odd case, we see that since $M$ contains transvections,
there are no examples with $M$ irreducible.
\end{proof}  

Next we consider subgroups $H$ which are not necessarily maximal but still
satisfy the hypotheses of Theorem \ref{simplereal}.   Since we only use the result
when $n$ is odd, we assume that and leave the case $n$ even to the interested reader.

\begin{cor}  \label{nonmax}  Let $L$ be a finite nonabelian simple group and $H$ a proper subgroup of $L$
such that $H$ intersects every  conjugacy class of real elements in $L$ and $[L:H]$
is odd.    Then
either $H$ is maximal  in $L$
or one of the following holds:
\begin{enumerate}
\item $L=M_{23}$ and there are three possible conjugacy classes of subgroups $H$
each self normalizing. In each case,  $(1_H)^L$ contains
a real-valued nontrivial constituent of multiplicity $1$;
 or 
\item $L=\PSL_n(q)$ with $n$ odd, $H$ is contained in a unique maximal
subgroup $M$ which is the stabilizer of a $1$-space or hyperplane and
$(1_H)^L$ contains a real-valued constituent of degree $(q^{n-1} - 1)/(q-1) - 1$ with multiplicity $1$.  
\end{enumerate}
\end{cor} 

\begin{proof}   If $L$ is a sporadic group, then by the previous result $L \cong M_{22}$ or $M_{23}$ and
one computes using GAP (again, we thank Thomas Breuer).

   Otherwise,  $L \cong \PSL_n(q)$ with $n$ odd.    Let $M$ be a maximal subgroup containing $H$
and we assume that $H \ne M$.  Then $M$ is the stabilizer of a $1$-space or hyperplane.  
  As we noted in the proof of the previous result,  $H$ contains a real element $x$ acting
irreducibly on a hyperplane.   Note that $x$ fixes a unique $1$-space and a unique hyperplane
which do not intersect.   Thus, if $H$ is contained in two maximal subgroups,  $H$ would be contained
in their intersection.  This intersection is the Levi subgroup preserving the $1$-space and hyperplane
preserved by $x$.   Since this subgroup contains no regular unipotent elements,  this is not possible.

By applying an outer automorphism if necessary, we may assume that $M$ is the stabilizer of $1$-space. 
Let $U$ be the unipotent radical of $M$.  
By Corollary \ref{c:unipotent},  $H$ intersects every conjugacy class of unipotent elements.   

If $q$ is odd, then considering real elements that have $-1$ as eigenvalue  odd multiplicity and arguing
as in the case of $\SL_{n-1}(q)$,  we see that $HU$ contains $[M,M]$.   If $H \cap U =1$, then as $H^1(\SL_{n-1}(q), U)=0$ \cite{JP}, 
$H$ would be a Levi complement but this is a contradiction as $H$ would not contain a regular unipotent element.
Thus, $H \ge [M,M]$ and $H$ has exactly $2$ orbits on $1$-spaces, whence the statement about the real constituent
follows.

Suppose that $q$ is even.   Since $H$ contains unipotent elements of every possible Jordan form
and $H$ intersects every conjugacy of real semisimple elements, then arguing as in the proof of
Theorem \ref{simplereal}, it follows that $HU/U \ge \Sp_{n-1}(q)$.  If $n \ge 9$,  then $\Sp_{n-1}(q)U$
does not contain any unipotent elements with exactly $3$ Jordan blocks with the blocks of odd distinct sizes
and so for $n \ge 9, HU \ge [M,M]$ and we argue as in the case $q$ is odd.    

In all cases we see that either $H \ge U$ or $H$ is isomorphic to a subgroup of $\GL_{n-1}(q)$.
The number of unipotent classes in that group is less than in $G$ and since $H$ intersects each
unipotent class, that is not possible.  Thus,  $H \ge U\Sp_{n-1}(q)$.   Since $\Sp_{n-1}(q)$ is transitive
on nonzero vectors, it follows that $H$ has two orbits on lines and so the statement
about the real constituent follows.  
\end{proof}

The almost simple case also follows easily:

\begin{thm} \label{almostsimplereal}  Let  $G$ be a finite almost simple group
with socle $L$  (a nonabelian simple group).   Let  $M$ be
a maximal subgroup of $G$ such that $[G:M]$ is odd and $D:=L \cap M \ne L$
intersects all real $G$-classes contained in $L$.   Then one of the following holds:
\begin{enumerate}
\item  $G=L =M_{22}$,  $M = 2^4:\Alt_6$,  $[G:M]=77$ and $(1_M)^L = 1a+21a+55a$; 
\item $G=L=M_{22}$,   $M=2^4:\sym_5$,  $[G:M]=231$ and $(1_M)^L =1a+21a+55a+154a$; 
\item  $G=L = M_{23}$, $M=M_{22}$,  $[G:M]=23$ and $(1_M)^L= 1a+22a$; 
\item  $G=L = M_{23}$, $M=\PSL_3(4).2_2$, $[G:M]=253$ and $(1_M)^L = 1a+22a+230a$; 
\item  $G=L = M_{23}$, $M=2^4:\Alt_7$, $[G:M]=253$ and   $(1_M)^L= 1a+22a+230a$; 
\item $G=L = M_{23}$, $M=2^4:(3 \times \Alt_5).2$, $[G:M]=1771$ and   $(1_M)^L= 1a+22a+230aa+253a+1035a$.
\item  $L=\PSL_n(q)$ with $n \ge 3$ and odd,   $M$ is the stabilizer of 
$1$-space or hyperplane, $[G:M]={q^n-1 \over q-1}$ and $L$ acts $2$-transitively on the cosets
of $M$.   Moreover, $[G:L]$ is odd unless possibly $L=PSL_3(4)$ and $G/\Oh{2'} G$ is generated by a field
automorphism of order $2$. 
\end{enumerate} 
\end{thm}  

\begin{proof}  If $G=L$, the previous result applies. 

Suppose that $G > L$.   If $L$ is sporadic, then $L=M_{22}$ and if an outer automorphism
is present, then elements of order $11$ in $L$ are real and we see there are no proper maximal
subgroups of $L$ containing all real elements (the only possibilities are the $D$ already mentioned
and they do not contain elements of order $11$).   

The only other possibility is that $L=\PSL_n(q)$ with $n$ odd and by the
case $G=L$ we see that $D$ is contained in a unique maximal subgroup $P$ of $L$, a  parabolic subgroup.   
Since $M=N_G(D)$, it follows that $M$ normalizes $P$ whence $D=P$.   
Since 
$[PGL_n(q):\PSL_n(q)] = \gcd(n,q-1)$ is odd, if $G/L$ has even order,   $G$ contains
an outer involution that is either a graph, field or graph-field automorphism.   Since
$G=LN_G(D)$,  $G$ cannot contain a graph or graph-field automorphism.  

Suppose that $G$ contains a field automorphism of order $2$.   Then  $q=q_0^2$. 
Aside from the case that $n=3$ and $q=4$, there exists a primitive prime divisor $p$
of $q^n-1$ and so $p$ divides $q_0^n +1$.  Thus, there is  $x \in \PSU_n(q_0) < L$ acting
irreducibly.   Since
a field automorphism induces the transpose inverse automorphism on $\PSU_n(q_0)$,  
$x$ is real in $G$, a contradiction.   

In the remaining case,  $L=\PSL_3(4)$ and $G \le P\Gamma L(3,4)$.   Any real element
of $G$ contained in $L$ does have a fixed point.  
\end{proof} 

We note that if $L=\PSL_3(4)$ and every real element of $G$ has a fixed point, then
either $[G:L]=3$ or $[G:L]=2$.

We next consider general primitive groups of odd degree. This is Theorem
C of the introduction.  Recall that real-valued irreducible characters are either of $+$ type (or orthogonal)
if they can be afforded by a representation with real entries, or of $-$ type
(or symplectic), if they cannot.

\begin{thm} \label{primitivereal}   Let  $G$ be a primitive permutation group acting on
a set $X$ of odd cardinality.    Assume that every real element of the socle of $G$ fixes a point
of $X$.   Let $M$ be a point stabilizer.  Then  $G$ has a unique minimal 
normal subgroup $A$ and either $G$ has odd order or 
$A = L_1 \times \ldots \times L_t$   with $L_i \cong L$ a nonabelian simple
group. Also,  $M \cap A = D_1 \times \ldots \times D_t$,  $D_i$
is maximal in $L_i$ (and is described in Theorem \ref{simplereal})
 and intersects every real conjugacy class of $G$ which
intersects $L_i$.   Furthermore 
$(1_M)^G$  has a nontrivial orthogonal constituent of multiplicity $1$.
Moreover either $G/A$ has odd order or $L \cong \PSL_3(4)$ and 
$G/A_1$ has odd order where $A_1/A$ is  an elementary abelian $2$-group
and $A_1$ normalizes each $L_i$.  
Indeed,  $L_i \cong M_{22}, M_{23}$ or $\PSL_n(q)$ with $n$ odd. 
\end{thm}

\begin{proof}  We have that $M$ is maximal in $G$ with 
${\rm core}_G(M)=1$. Since $M$ contains a Sylow $2$-subgroup  $S$ of $G$,
we have that 
$\oh 2 G=1$.   Suppose that $\oh pG \ne 1$ with $p$ odd.   
Let $A$ be a minimal normal subgroup contained in $\oh pG$. Then
$G=AM$  and $A \cap M=1$.   If $G$ does not have odd order, then
$S$ cannot centralize $A$, and therefore there is $1 \ne a \in A$
which is inverted by some element of $S$ (see for instance Lemma 3.1(d) of \cite{DMN}).
By hypothesis,  some $G$-conjugate of $a$ is in $M$, but this is impossible since $A\cap M=1$.

Let $A$ be a minimal normal subgroup of $G$.
Thus, $A$ is a direct product  $L_1 \times \ldots \times L_t$
with $L_i \cong L$ a nonabelian simple group.  Since $L_i$ contains real
elements, $D_i=M \cap L_i \ne 1$ for some $i$.   
If $B$ is another minimal normal subgroup of $G$, then $[A,B]=1$
and thus $M\cap B$ is normal in $AM=G$. Hence $M\cap B=1$ and therefore
$|B|=[G:M]$ is odd. But this cannot happen. Hence $A$ is the unique
minimal normal subgroup of $G$.

We can view $X$ as $L_1/D_1 \times \ldots L_t/D_t$ and so we see that $x \in L_1$
fixes a point on $X$ if and only if it fixes a point on $L_1/D_1$.    
It follows that $D_i$ intersects any real class of $G$ that intersects
some $L_j$.   Let $G_i = \norm G{L_i}/\cent G{L_i}$.
This is an almost simple group with socle $L_i$ and $\norm{G_i}{D_i}$ is a maximal
subgroup of $G_i$.   Our assumptions imply that $D_i$ intersects every real
class of $L_i$ (in $G_i)$ whence $(G_i, L_i, \norm {G_i}{D_i}, D_i)$ satisfies the
conclusions of Theorem \ref{simplereal}.  In particular, $D_i$ is maximal in $L_i$
and $G_i/L_i$ has odd order unless $L_i \cong \PSL_3(4)$.    

We next show that $G/A$ has odd order if $L_i \ne \PSL_3(4)$.   Suppose that there exists
$g \in G \setminus{A}$ an outer automorphim with $g^2 \in A$.  
We may assume that $g$ has order a power of $2$.  Then
$g$ cannot normalize all components $L_i$ (as the automizer in $G$ of any
component has odd order).   Thus,   we may assume that $g$ interchanges $L_1$ and $L_2$
and $g^2$ induces inner automorphisms on $L_1$ and $L_2$.  Replacing $g$ by $gu$
with $u \in L_1 \times L_2$ allows us to assume that $g$ interchanges the coordinates
of $L_1$ and $L_2$.  
   Choose $x \in L_1$ that has no fixed points. 
 Then $g$ inverts  $c:=(x,x^{-1}, 1 \ldots, 1)$ with $c \in A$.   Then $c$ has no fixed points
but is real, a contradiction.    A similar analysis when $L_i \cong \PSL_3(4)$ completes the proof. 
\end{proof}

Here is the analog for $n$ even. 

\begin{thm} \label{primitiverealeven}   Let  $G$ be a primitive permutation group acting on
a set $X$ of even cardinality $n$.    Assume that every real element of the socle of $G$ fixes a point
of $X$.   Let $M$ be a point stabilizer.  Then  $G$ has a unique minimal 
normal subgroup $A$ and 
$A = L_1 \times \ldots \times L_t$   with $L_i \cong L$ a nonabelian simple
group. Also,  $M \cap A = M_1 \times \ldots \times M_t$,  $M_i$
is maximal in $L_i$ and intersects every real conjugacy class of $G$ which
intersects $L_i$.   Furthermore $G/A$ has odd order and $(L_i, M_i)$
are described in Theorem \ref{simplerealeven}.   
\end{thm}

\begin{proof}     Let $A$ be a minimal normal subgroup of $G$.  Then $A$ acts transitively
whence $A$ has even order.   If $A$ acts regularly, then some involution of $A$
has no fixed points, a contradiction.   Thus, $G$ has no regular normal subgroups
and so $A = L_1 \times \ldots \times L_t$ with $L_i \cong L$, a nonabelian
simple group.   Since $M \cap L_i \ne 1$ for some $i$ (since $L_j$ contains real
elements), it follows that $M \cap A = M_1 \times \ldots \times M_t$ where 
$M_i$ are maximal in  the almost simple group $\norm G{M_i}/\cent G{M_i}$.   It follows
by Theorem \ref{simplerealeven} that $L_i$ is a Mathieu group and has trivial
outer automorphism group, whence $\norm G{M_i}/\cent G{M_i} = M_i$ and $M_i$
is maximal in $L_i$.  Arguing as in the case that $n$ is odd shows that 
$G/A$ has odd order.
\end{proof}

We close this section by proving a result about transitive groups which we require
for $2$-Brauer characters.

\begin{cor} \label{cor:oddreal}   Let $G$ be a (faithful)  transitive  permutation group of odd degree $n$.  
Let $A = L_1 \times \ldots  \times L_t$ be a normal subgroup of $G$ such that each $L_i \cong L$,
a nonabelian simple group  not isomorphic to any of $M_{22}, M_{23}$ or $\PSL_n(q)$ with $n$ odd.
Then some real element (in $G$) of $A$ is fixed point free.
\end{cor}

\begin{proof}   Let $H$ be a point stabilizer and set $B = H \cap A \ne A$.   Since $n$ is odd,
$ B \ge S_1 \times \ldots \times S_t$ where $S_i \in \Syl_2(L_i)$.   Let $Q_i = \pi_i(B)$
where $\pi_i$ is the projection from $A$ onto $L_i$.   If $Q_i = L_i$, then
$L_i = [Q_i, S_i]=[B, S_i] \le B$.   Thus (reordering if necessary),  $B \le  M_1 \times L_2 \times
\ldots \times L_t$ where $M_1$ is a maximal subgroup of $L_1$ containing $Q_1$.
    Let $a \in L$ be a real element.   Then $y:=(a,\ldots, a)$ has a fixed point, whence
$y^G \cap H$ is nonempty.   This implies that $M_1 \ge Q_1$ contains an $L_1$ conjugate of 
$a$ of every real element in $L_1$.   By Theorem \ref{simplereal} and the fact that we are excluding
the simple groups occurring in the conclusion of that theorem, we obtain a contradiction.
\end{proof} 

\section{Sylow 2-subgroups and Theorem D}

In this section we prove Theorem D.
For Brauer characters we use the notation in \cite{N}.
If $\Psi$ is a Brauer character of $G$, and
we write $\Psi=\sum_{\varphi \in \ibr G} a_\varphi \varphi$,
we call $a_\varphi$ the multiplicity of $\varphi$ in $\Psi$.

\medskip
Recall that the equivalence of (i) and (ii) in Theorem D is elementary
and well-known. (See for instance Proposition 6.4 in \cite{DNT}.)
We start with condition (iii).

%Lemma 3.1(d) of \cite{DMN}

\medskip

\begin{thm}\label{cor}
Let $G$ be a finite group and assume that $P \in \syl 2 G$ is not normal in $G$. 
 Then the 2-Brauer character $({1_{\norm GP}}^0)^G$ contains
 a non-trivial irreducible real 2-Brauer character with odd multiplicity.
\end{thm}

\begin{proof}
Arguing by contradiction, suppose that
$$({1_{\norm GP}}^0)^G=a{1_G}^0 + \sum_i b_i(\varphi_i + \bar\varphi_i) +
\sum_j c_j \mu_j ,$$
where $\varphi_i \in \ibr G$ are non-real, $\mu_j \in \ibr G$ are real but non-trivial,
$a$ is a positive integer, 
and $c_j$ is even for every $j$.
Now, by Theorem 2.30 of \cite{N}, we have that $\mu_j(1)$  is
even for every $j$.   Since $[G:\norm GP]$ is odd, we conclude that $a$ is odd.

Let $1\ne x \in G$ be a real element of odd order (using that $P$ is not normal
and Proposition 6.4 in \cite{DNT}).
Suppose that $x$ normalizes
some Sylow 2-subgroup $Q$. Let $y \in G$ invert $x$ with $y$ a 2-element.
Let $\Omega$ be the set of Sylow 2-subgroups of $G$ normalized
by $x$. Then notice that
$\langle y\rangle$ acts on $\Omega$  and has no fixed
point on it. Indeed, if $Q\in \Omega$ and
 $Q^y=Q$, then $y \in Q$. Then $y^x \in Q^x=Q$. However,
 $y^xy^{-1}=x^{-2} \in Q$, a contradiction. We conclude that $|\Omega|$ is even.
 Therefore $b=({1_{\norm GP}}^0)^G(x)$ is even.
 Thus
$$b=({1_{\norm GP}}^0)^G(x)= a + 2 \alpha $$
for some algebraic integer $\alpha$. We deduce that $-a/2$ is an algebraic integer,
but this is not possible because $a$ is odd.
\end{proof}

We briefly digress to   recall
 that, unlike ordinary characters,
  the trivial 2-Brauer character does not necessarily appear
with multiplicity 1 in the 2-Brauer character ${((1_P)^0})^G$,
 even if $P$ is self-normalizing. In the following, 
  $\Phi_\varphi$ is the projective indecomposable
character associated with $\varphi \in \ibr G$.

  \begin{lem}\label{BigP}
 If $p$ is a prime, $G$ is a finite group and $P\in \syl pG$, then
 $$({1_{P}}^0)^G=\sum_{\varphi \in \ibr G} {\Phi_\varphi(1) \over |P|} \varphi \, \, . $$
  \end{lem}
\begin{proof}
Use the induction formula and compare the character $({1_{P}}^0)^G$
with the restriction of the
regular character of $G$ to the $p$-regular elements of $G$.
\end{proof}
(Notice that Lemma \ref{BigP} reproves the well-known fact that $|G|_p$ divides the degrees of the projective indecomposable
characters.) Now, if $G= \alt 6$, then $P=\norm GP$ and the 2-Brauer trivial character of $G$ enters with
multiplicy 5 in the Brauer character $({1_{P}}^0)^G$.
\medskip

Next we show that condition (iv) of Theorem D does not happen in
 simple groups.

\begin{thm}\label{simplecase}   Let $G$ be a finite nonabelian simple group and let $P$ be a Sylow $2$-subgroup
of $G$.  Then there exists $g \in G$ with $g$ real and of odd order with $g$ not normalizing
any $G$-conjugate of $P$.
\end{thm} 

\begin{proof}  Let us remark that if $P=\norm GP$, it follows by the Baer-Suzuki theorem that 
there  exists a real element $g \in G$ of odd order, whence the result follows.  This handles
$21$ of the $26$ sporadic groups.  In the remaining $5$ cases,  there is a real element of order
$5$ and $5$ does not divide the order of $\norm GP$ (these are the cases $G=J_1, J_2, J_3, Sz$ and   $HN$).

If $G=\Alt_n, n \ge 6$, then   $P=\norm GP$ and the result follows.   If $n=5$, then $5$-cycles are real
and do not normalize a Sylow $2$-subgroup. 

Now suppose that $G$ is a simple group of Lie type.   Since the proof is quite similar to that of 
Theorem \ref{simplereal}, we just indicate the modifications required.

Suppose first that $G$ is a group in characteristic $2$.   In this case, the proof given in Theorem \ref{simplereal}
shows that there exist odd order real elements not contained in any parabolic subgroup aside from the
cases with $G = \PSL_n(q)$ with $n$ odd.   The proof shows that the only parabolics which 
intersect every real class are the two maximal end node parabolics and in particular not the Borel subgroup,
which is the normalizer of a Sylow $2$-subgroup.

Finally, suppose that $G$ is a simple group of Lie type in characteristic not $2$.  If $G = \PSL(2,q), q \ge 11$, 
then either $P$ is self normalizing and the result follows as above or $|P|=4$ and $|\norm GP|=12$.  Note that
there is a real element of odd order either $(q+1)/2$ or $(q-1)/2$ and this element does not normalize a
Sylow $2$-subgroup.

 If $G=\PSL_n(q)$ or $\PSU_n(q)$ with $n \ge 3$, then transvections are real.   Let $V$ be the natural module
 for $\SL_n(q)$ or $\SU_n(q)$.    Note that a transvection
normalizing a Sylow $2$-subgroup would have to preserve each  irreducible $P$-submodule of $V$ of dimension
at least $2$.   Thus, if  unless $P$ is diagonalizable,  we can reduce to the  case that $P$ is irreducible.   
However,  by \cite{kantor}, the   irreducible subgroups containing
transvections are classified and $\norm GP$ is not a possibility.      If $P$ is diagonalizable, then by induction,
we can reduce to the case $n=2$ which was handled above.

Next assume that $-1$ is in the Weyl group of $G$.   Then every semisimple element is real.  Note that there
exist unipotent real elements  (since a Borel subgroup  $B$ has even order and $\oh 2B=1$, this
follows by the Baer-Suzuki theorem).   It follows that any prime dividing $|G|$ also divides $|\norm GP|$.
By \cite[10.7]{LPS},  the only possibilities for $(G,M)$ with $M$ containing $\norm GP$ are
$G=\Omega_{2m+1}(q)$ and $M$ the stabilizer of a hyperplane of $-$ type.     Let 
$x$ be an element whose order is the odd part of $q^m-1$.  Then the only irreducible subspaces of $x$
are two totally singular hyperplanes of dimension $m$ and a nondegenerate $1$-space.  Note that
$x$ is real and not conjugate to an element of $M$, whence the result. 

The only remaining possibility for $G$ a classical group is $\Omega_{2m}^{\pm}(q)$ with $m \ge 5$ odd.  
Note that in this case $P$ has a unique two dimensional irreducible submodule and so we can work
in $\Omega_{2m-2}^{\pm}(q)$ and use the previous result.

This leaves only the case that $G=E_6^{\pm}(q)$.  It is easy to see that $\norm GP/P$ embeds in the Weyl
group and so in particular contains no  elements of prime order greater than $5$.  On the other
hand, if $g \in F_4(q) < G$ is an element whose order is a primitive prime divisor of $q^{12}-1$, then
$g$ is real and has order at least $13$, whence $g$ normalizes no conjugate of $P$.  
\end{proof}

In order to prove the next result,  which easily closes the proof of Theorem
D, we use the fact that odd order real elements of factor
groups can be lifted to real elements of the group. This is key in an inductive hypothesis.

\begin{thm}\label{real}
Suppose that $G$ is a finite group,
 $P \in \syl 2 G$.
 If every odd order real element of $G$ lies in some $G$-conjugate of $\norm GP$,
 then $P$ is normal in $G$.
 \end{thm}
 
 \begin{proof}
 Argue by induction on
 $|G|$.  Suppose that $1<N$ is normal in 
 $G$, and let $Nx \in G/N$ of odd order.  By Lemma 3.2 of \cite{NST},
 there is $y \in G$ real such that $Nx=Ny$. Since $Nx$ has odd order, then
 by taking $y_{2'}$, we may assume that $y$ has odd order. Then $y^g \in \norm GP$ for some $g \in G$, 
 so $Nx^g \in \norm{G/N}{PN/N}$, and by induction, we may assume that
 $PN$ is normal in $G$. 
 
 In particular, assume that $N$ is a minimal normal subgroup, and let $K=PN$.  If $N$ is a $2$-group, then $N \le P$
 and so $P$ is normal in $G$.   If $N$ is an elementary abelian subgroup of odd order, then either
 $P$ centralizes $N$ whence $P$ is characteristic in $K$ and so normal in $G$ or $\cent PN=1$
 and $P=\norm KP$.   Then Baer-Suzuki implies that there exist a real element $g$ of odd order in $N$.
 Note that $g$ does not normalize a Sylow $2$-subgroup of $K$.    Since $K$ is normal in $G$, 
 any Sylow $2$-subgroup of $G$ is contained in $K$ and we have a contradiction.

 So $N$ is a direct product of nonabelian simple groups.   Let $L$ be a component of $N$.
 By Theorem \ref{simplecase},  there is an element $g \in L$ that is real and of odd order which
 does not normalize any Sylow $2$-subgroup of $L$.   Then $g$ normalizes no Sylow $2$-subgroup
 of $N$ (since if $g$ normalizes some Sylow $2$-subgroup $T$ of $G$, then $g$ also normalizes 
 $T \cap L$, a Sylow $2$-subgroup of $L$).  
 \end{proof}

 \section{Proof of Theorem B}
Our notation for characters follows \cite{Is}. If $N\nor G$ and $\theta \in \irr N$,
then $\irr{G|\theta}$ is the set of irreducible
constituents of the induced character $\theta^G$. By the  Frobenius reciprocity, these
are the characters $\chi$ such that $\theta$ is a constituent of the restriction
$\chi_N$. Also, $\nu_2(\chi)$ denotes the Frobenius-Schur indicator of $\chi \in \irr G$.
If $\chi$ is real valued, then $\nu_2(\chi)=\pm 1$, according to its type.

\begin{lem}\label{intersect}
 Let $G$ be a finite group, and let $H$ be a subgroup of $G$. If
    $(1_H)^G$ has
a unique real-valued irreducible constituent with odd multiplicity, then
$x^G \cap H \ne \emptyset$ for every real element $x \in G$. Furthermore,
$[G:H]$ is odd.
\end{lem}

\begin{proof}
Write
$$(1_H)^G=1_G+2\sum_{i} a_i\epsilon_i + \sum_j  b_j(\tau_j + \bar\tau_j) \, ,$$
where $\epsilon_i \in \irr G$ are real-valued, and $\tau_j \in \irr G$ are not.
If $x \in G$ is real and $x^G \cap H=\emptyset$, then by evaluating in $x$ the expression above,
we conclude that 
$$0=1 +2 \alpha \, ,$$
where $\alpha$ is an algebraic integer. This is impossible. For the second part,
evaluate the character $(1_H)^G$ in the trivial element.
\end{proof}

Indeed, in Lemma \ref{intersect}, the conclusion  $x^G \cap H \ne \emptyset$ can be replaced by the stronger
condition $(1_H)^G(x)$ is odd for every real element $x \in G$.  Unfortunately,
this does not seem to shorten the proofs much -- in all the cases with 
$G$ primitive of odd degree, the condition that
every real element  has a fixed point turns out to be equivalent 
to the condition that every real element has an odd number of fixed points. 

\medskip

%\begin{thm}\label{simple}
%Suppose that $G$ is a simple and let $H$ be
%a maximal subgroup of $G$ of odd index.
%Assume that every real conjugacy class of $G$ intersects with $H$.
%Then there exists a real-valued $1 \ne \theta \in \irr G$ over $1_H$ 
%such that $[\theta_H, 1_H]$ is odd.
%\end{thm}
%

The following is well-known.
\begin{lem}\label{belowover}
Suppose that $G/N$ has odd order. 
\begin{enumerate}
\item
If $\chi \in \irr G$ is real valued, then every irreducible
constituent of $\chi_N$ is real-valued.
\item If $\theta \in \irr N$ is real-valued, then
there exists a unique  real-valued $\chi \in \irr{G|\theta}$. 
\end{enumerate}

\end{lem}

\begin{proof}Let $\theta \in \irr N$ be under $\chi$. 
By Clifford's theorem, we have that
 $\bar\theta$ and $\theta$ are $G$-conjugate by some
 $g \in G$. Then $g^2$ fixes $\theta$, and therefore $g$ fixes $\theta$,
 using that $G/N$ has odd order.   Hence, $\theta=\theta^g=\bar\theta$.
 The second part is  Corollary 2.2 of \cite{NT}.
\end{proof}

%\begin{lem}\label{same}
%Suppose that $\chi \in \irr G$ is real-valued and that $\chi_H=\nu \in \irr H$.
%Then $\chi$ and $\nu$ are of the same type.
%%\end{lem}
%
%\begin{proof}
%This is an easy application of Theorem 4.14 of \cite{Is}.
%\end{proof}
  
\medskip
We should warn the reader that if $[G:H]$ is odd and $\theta \in \irr H$ is real
valued, even of $+$ type, it is not necessarily true that $\theta^G$ contains 
a real-valued irreducible constituent. The 
${\tt SmallGroup}(48,30)=C_3:Q_8$ is an example.
Also, it is often the case, that odd-degree real-valued irreducible constituents of $(1_H)^G$ appear
with odd multiplicity whenever $[G:H]$ is odd. But not always, as shown by $G={\rm PSL}_2(25)$ and $H=D_{24}$.

\medskip

 The following includes Theorem B of the introduction.

\begin{thm}\label{oddmult}  Let $G$ be a finite group and let
$H$ be a  proper subgroup of $G$.   Assume that   
  $G=KH$ where $K$ is the smallest normal subgroup of $G$ with $G/K$ of odd order.
Then $(1_H)^G$ contains a non-trivial real-valued irreducible
 constituent with  odd multiplicity.
\end{thm} 

\begin{proof}  
Let $G$ be a counterexample minimizing $|G|$.
 By Lemma \ref{intersect}, we have that $[G:H]$ is odd and that $H$ intersects
every real conjugacy class of $G$.   Let $1<V$ be a normal subgroup of $G$.
By working in $G/V$, we may assume that $HV=G$ (since if $\phi$ is an
irreducible constituent of $(1_{HV})^G$, then its multiplicity in $(1_{HV})^G$ is the same as in $(1_H)^G$).  In particular, 
${\rm core}_G(H)=1$. 

 Note that $\oh 2 G\le H$ and so
$\oh 2 G=1$.    Let $A$ be a minimal normal subgroup of $G$, so that $G=HA$.  
Suppose that $A$ is an elementary abelian $p$-group for some odd prime $p$. 
Then $H\cap A \nor G$ and thus $H\cap A=1$ and $\cent HA=1$.
In particular, if $P$ is a Sylow 2-subgroup of $H$, then $P$ does not centralize $A$.
Hence, $A$ contains a non-trivial real element $x$ (see for instance Lemma 3.1(d) of
\cite{DMN}). However, $x^G \cap H \sbs A \cap H=1$, a contradiction.

We claim that $A$ is the unique minimal normal subgroup of $G$.    Let $B$ be another such.
Then $G=HB$ and $B$ centralizes $H \cap A$, whence $H \cap A$ is normal, a contradiction.

 It follows that $A=L_1 \times \ldots \times L_t$ with $L_i \cong L$ a nonabelian simple group.
Let $C_i = H \cap L_i$. If $x \in L_i$ is real, then $x^g \in H$. Using that $g=ah$
for some $a \in A$ and $h \in H$, and that $[L_i,L_j]=1$ for $i\ne j$, we 
have that $x^{l_i} \in C_i$ for some $l_i \in L_i$. Therefore
$L_i$ satisfies the hypotheses of Corollary \ref{nonmax} with respect to the subgroup
$C_i$. 

Let $M$ be a maximal subgroup of $G$ containing $A \cap H$.  
Notice that ${\rm core}_G(M)=1$, since otherwise, $A \le M$ but $G=AH$.
By Theorem \ref{primitivereal}, we have that  $K=A$ unless possibly $L \cong \PSL_3(4)$
and $K/A$ is an elementary abelian $2$-group normalizing each $L_i$.  
In particular, $H$ acts transitively on 
the $L_i$.

Let $C=C_1 \times \ldots \times C_t$ and note that $H \le \norm GC$.  
Also, $\norm AC=B_1 \times \ldots \times B_t$, where $B_i=\norm{L_i}{C_i}$.
Let $\pi_i: H \cap A \rightarrow L_i$ be the projection map.
So $C_i \le \pi_i(H \cap C) \le B_i$.   We now apply Corollary \ref{nonmax}.    If $L = M_{22}$,
then   $B_i=C_i$ is maximal in $L_i$.   If $L = M_{23}$, then $B_i=C_i$ since $C_i$ is self-normalizing.
Thus, in these cases  $A \cap H = C$.   

In the case that $L = \PSL_n(q)$ with $n$ odd,  $B_i$ is the maximal parabolic containing $C_i$.   
  
  By Corollary \ref{nonmax},
  we have that
  the characters 
 $$(1_{C_1})^{L_1} \times \ldots \times (1_{C_t})^{L_t}$$
 and $$(1_{B_1})^{L_1} \times \ldots \times (1_{B_t})^{L_t}$$
 have a non-trivial
real-valued constituent $\eta \in \irr K$ of multiplicity $1$. 
Since $C \le H\cap K \le B$, in all cases we conclude that
    $$((1_H)^G)_K=(1_{H\cap K})^K$$
   contains a non-trivial irreducible constituent $\eta$ with multiplicity 1.
   (Here we are using that $(1_{C_i})^{L_i}$ and
   $(1_{B_i})^{L_i}$ contain the same non-trivial irreducible constituent $\tau_i$
   with multiplicity 1, by Corollary \ref{nonmax}.)
   
Since $G/K$ is odd, let $\nu \in \irr G$ be the unique real-valued irreducible
constituent lying over $\eta$, by Lemma \ref{belowover}.
   Now, if we write
   $$(1_H)^G=a\nu + \sum_{i} a_i (\nu_i + \bar\nu_i) + \Delta \, $$
   where $\Delta$ is a character such that $[\Delta_K, \eta]=0$ and $\nu_i \in \irr G$
   are non-real and lie over $\eta$, we have that
   $$1=[((1_H)^G)_K, \eta]=a + 2\sum_i a_i[\nu_i, \eta], \, $$
   and we deduce that $a=1$.    \end{proof}   
   
   Notice that Lemma \ref{bob} below implies that the odd-multiplicity irreducible
   character in Theorem \ref{oddmult} is of $+$ type. The following is well-known.
   
\begin{lem}\label{bob}
Suppose that $G$ is a finite group, and $H$ is a subgroup
of $G$.   Let  $\rho$ be a character of $G$ which admits a $G$-invariant
non degenerate quadratic form.  If $\chi \in \irr G$ is real-valued 
and   $[\chi_H, \rho]$ is odd, then $\chi$ has +  type.
\end{lem}

\begin{proof}
Let $V$ be a $\mathbb{C}G$-module affording $\rho$.  Let $W$
be the homogeneous component of $V$ corresponding to $\chi$.
Note that $W$ is orthogonal to every other homogeneous component
(since it is self dual), whence $W$ is non degenerate.  

So we may
assume that $\rho = m \chi$ with $m$ odd.  We induct on $m$. If $m=1$,
the result is clear.  Otherwise,  let $U$ be an irreducible
submodule of $W$.   If $\nu_2(\chi)=-1$,  then $U$ must be totally
singular with respect to the quadratic form.     Then $U^{\perp}/U$ is non degenerate with respect
to the quadratic form  with character $(m-2)\chi$ and the result follows. 
\end{proof}
\medskip

In particular, Lemma \ref{bob} applies to the case where
$\rho=(1_H)^G$ or more
generally if $\rho= \psi^G$ with $\psi \in \irr G$ irreducible of  + type
 under the
assumption that $[\rho_H, \psi]$ is odd.  One can also give a proof
of the previous result using Schur indices.   The proof above
shows that the Schur index over $\RR$ is odd and divides $2$
whence the result.

   \medskip
   To summarize the results in this section,
    we have shown the following.
   \begin{thm}\label{all}
Let $H$ be a subgroup of $G$.   Then $(1_H)^G$ has a
non-trivial real-valued irreducible constituent 
of odd multiplicity under any of the following  hypotheses.
\begin{enumerate}
\item  $[G:H]$ is even.
\item  $H \cap x^G = \emptyset$ for some real element $x$ of $G$.
\item  $(1_H)^G(x)$ is even for some real element $x$ of $G$.
\item  $H$ is maximal in $G$ and $G/{\rm core}_G(H)$ has even order. 
\item  $G=\Oh{2'} GH$ and $H$ does not contain $\Oh{2'} G$.
\end{enumerate}
\end{thm}
\begin{proof} The first three results follow from Lemma \ref{intersect}
and its proof,
while (iv) and (v) follow from Theorem \ref{oddmult}.
\end{proof}

\section{More on 2-Brauer characters} \label{sec:2brauer}

In this final section, we prove that there are versions of Theorems A and B
for 2-Brauer characters. First we need some results for $p$-Brauer
characters, where $p$ is any prime, 
which are essentially well-known. 
If  $\Psi=\sum_{\varphi \in \ibr G} a_\varphi \varphi$
is a Brauer character of $G$,
 with certain abuse of notation, let us write $a_\varphi=[\Psi,\varphi]$.

\begin{lem}\label{belowoverBr}
Suppose that $G/N$ has odd order. 

\begin{enumerate}
\item If $\chi \in \ibr G$ is real valued, then every irreducible
constituent of $\chi_N$ is real-valued.

\item  If $\theta \in \ibr N$ is real-valued, then
there exists a unique real-valued $\chi \in \ibr{G|\theta}$. Furthermore,
$[\chi_N,\theta]=1$. 
\end{enumerate}
\end{lem}

\begin{proof} 
Part (i) follows precisely as in the proof of  Lemma \ref{belowover}(i).
The first part of (ii) follows from Lemma 5 and Corollary 1 of \cite{NT06}.
 Now, notice,
in that proof, that  $\chi=\psi^G$, where $\psi \in \irr T$ is the unique real extension
of $\theta$ to the stabilizer $T$ of $\theta$ in $G$.
By the Clifford correspondence for Brauer
characters (Theorem 8.9 of \cite{N}),
we have that $[\chi_N,\theta]=[\psi_N,\theta]=1$.
\end{proof}
 
\medskip

We also need the following result, which is well-known for ordinary
characters. We have been unable to find a proof without using projective representations. (The analogs of this result for induction of 
Brauer characters and Isaacs $\pi$-partial characters can be found in \cite{NRS}.)

\begin{lem}\label{proj}
 Suppose that $G=NH$, where $N \nor G$ and $H \le G$.
 Let $D=N\cap H$.
 Let $\theta \in \ibr N$ be $G$-invariant such that $\theta_D=\varphi \in \ibr D$.
 Then restriction defines a bijection
 $\ibr{G|\theta} \rightarrow \ibr{H|\varphi}$.
 \end{lem}

\begin{proof}
Let ${\mathcal P}$ be a projective representation
 of $G$ associated with $\theta$ with factor set $\alpha$.
 (That is, ${\mathcal P}_N$ affords $\theta$, ${\mathcal P}(gn)={\mathcal P}(g){\mathcal P}(n)$
 and ${\mathcal P}(ng)={\mathcal P}(n){\mathcal P}(g)$ for $g \in G, n \in N$,
 see Theorem 8.14 of \cite{N}.)
 We know that $\alpha(xn,ym)=\alpha(x,y)$ for $x, y \in G$
 and $n,m \in N$. View $\alpha \in {\bf Z}^2(G/N,\C^\times)$.
 Then notice that ${\mathcal P}_H$ is a projective representation
 of $H$ associated with $\varphi$.
 
 Now, let $\mu \in \ibr{G|\theta}$. By  
 Theorem 8.16 in \cite{N},  there
 is a projective representation $\mathcal Q$ of $G/N$ with
 factor set $\alpha^{-1}$ such that
 ${\mathcal Q} \otimes {\mathcal P}$ affords $\mu$.
 
 By Theorem 8.18 of \cite{N}, we have that $\mathcal Q$ is irreducible.
 Now, notice that $\mathcal Q_H$ is irreducible since $Nh \mapsto Dh$
 is an isomorphism $G/N \rightarrow H/D$.
 By Theorem 8.18 of \cite{N}, we have that the representation
 ${\mathcal Q}_H \otimes {\mathcal P}_H=
 ({\mathcal Q}  \otimes {\mathcal P})_H$ is irreducible. Therefore 
 $\mu_H$ is irreducible.
 
 Suppose now that $\tau \in \ibr{G|\theta}$, and let ${\mathcal Q}'$
 as before such that ${\mathcal Q}' \otimes {\mathcal P}$ affords $\tau$.
 If $\tau_H=\mu_H$ then
 ${\mathcal Q}'_H \otimes {\mathcal P}_H$ and 
 ${\mathcal Q}_H \otimes {\mathcal P}_H$ are similar. By Theorem 8.16 of \cite{N},
 ${\mathcal Q}'_H$ and ${\mathcal Q}_H$ are similar. By the isomorphism,
  ${\mathcal Q}'$ and ${\mathcal Q}$ are similar, and therefore so are
  ${\mathcal Q} \otimes {\mathcal P}$ and ${\mathcal Q}' \otimes {\mathcal P}$.
  Hence $\tau=\mu$. Surjectivity is proved similarly.
  \end{proof}

\begin{lem}\label{below2Br}
Suppose that $N \le H \le G$, where $N$ is normal in $G$, $G/N$ has odd order, and $\theta \in \ibr H$ is real-valued.
Then there is a unique real-valued $\chi \in \ibr G$  such that
 $\theta$ is an irreducible constituent of $\chi_H$.
Furthermore, $[\chi_H, \theta]=1$.
If $p=2$, then $[\theta^G, \chi]=1$.
 \end{lem}

\begin{proof}
By Lemma \ref{belowoverBr}(i), let $\nu \in \irr N$ be real-valued under $\theta$.
First of all notice that if $\chi$ exists necessarily it is unique.
Indeed, if $\chi_i \in \irr G$ are real-valued over $\theta$, then
$\chi_i$ lies over $\nu$, and therefore $\chi_1=\chi_2$ by Lemma \ref{belowoverBr}(ii).
Also, the fact that if $\chi$ exists then $[\chi_H, \theta]=1$,   follows
from the second part of Lemma \ref{belowoverBr}.  

In this paragraph, we prove that if $p=2$, then $\chi$ is 
necessarily an irreducible constituent of  
 $\theta^G$. By the second part of Corollary 8.7 of \cite{N}, we
 have that $[\chi_N, \nu]=[\nu^G, \chi]$, that is, Frobenius reciprocity holds for $N \nor G$ if the characteristic does not divide $|G/N|$.
 Hence, by Lemma \ref{belowoverBr}, we can write $\nu^G=\theta + \sum_{i} a_i(\eta_i + \bar\eta_i)$,
 where $\eta_i \in \ibr H$ are not real-valued. 
 Now $\nu^G=\theta^G +  \sum_{i} a_i(\eta_i + \bar\eta_i)$. 
 Since $\chi$ is real-valued and appears with multiplicity one in $\nu^G$,
 necessarily $\chi$ is a constituent of $\theta^G$, and appears with multiplicity $1$.

Hence, we only need to prove that there exists some 
real-valued $\chi \in \irr G$
 lying over $\theta$.
Suppose that $(G,H,N)$ is a counterexample minimizing first $|G|$, then
$[G:N]$, and finally $[G:H]$. 
Using our inductive hypotheses,
it is straightforward to check that we may assume that
 $H$ is maximal and that $N$ is the core of $H$ in $G$.
  If $H=N$, then the claim is Lemma \ref{belowoverBr}.
Now, let $K/N$ be a chief factor of $G$. Then $KH=G$ and $K\cap H=N$, using that
$G/N$ is solvable.
Let $I$ be the stabilizer of $\nu$ in $H$ and let $\tau \in \irr I$ be the Clifford correspondent
of $\theta$ over $\nu$.  By the uniqueness in the Clifford correspondence,
we have that  $\tau$
is real-valued. If $IK<G$, then
by induction, there is a real-valued $\mu \in \irr{IK}$ over $\tau$.
Since $[G:K]<[G:N]$, by induction,
there is $\chi \in \irr G$ real-valued over $\mu$.
Then $\chi$ lies over $\tau$. Thus there is $\epsilon \in \ibr H$
under $\chi$ such that $\epsilon_I$ contains $\tau$.  
By the Clifford correspondence for Brauer characters,
necessarily $\epsilon=\tau^H=\theta$, and we are done in this case.
  Hence, we may assume that $IK=G$
and therefore  $I=H$. If $H$ is the stabilizer of $\nu$ in $G$,
then $\theta^G$ is irreducible and lies over $\theta$
by the Clifford correspondence, and we are done. 
Hence, we may assume that $\nu$ is $G$-invariant. 
By Lemma \ref{belowoverBr}(ii), there is a unique 
real-valued $\hat\nu \in \irr K$  over $\nu$, that
also extends $\nu$. By uniqueness, notice that
$\hat\nu$ is $G$-invariant. By Lemma \ref{proj},
we have that restriction defines
a bijection $\irr{G|\hat\nu} \rightarrow \irr{H|\nu}$. Now,
if $\gamma \in  \irr{G|\hat\nu}$ extends $\theta$,
then $\gamma$ is real-valued by the uniqueness of the restriction map. 
\end{proof}

We have mentioned in the introduction that 
the conclusion of Theorem B does not hold for
Brauer characters, and we have given some counterexamples.
However, using similar arguments as those used before, we can prove that essentially these
are all the counterexamples.   
We only sketch a proof. 

\begin{thm}\label{mainbrauers}
Let $G$ be a group with a proper subgroup $H$ of odd index.  Set $K=\Oh{2'}G$
and assume that $H$ does not contain $K$.  
If $K$ does not surject onto $M_{23}$ or ${\rm PSL}_n(q)$ with
$n$ odd and $q$ even, 
then  the $2$-Brauer character $((1_H)^0)^G$ has a 
non-trivial real-valued irreducible constituent  $\phi$.
\end{thm}

\begin{proof}  First assume that $G=KH$.  

Arguing by induction on $|G|$ and then on the index $[G:H]$, it suffices to assume that $H$ is maximal in $G$
(note that this maximal subgroup does not contain $K$). 
In this case, we will prove that there is a real constituent of odd multiplicity.

We may assume that $H$ is corefree.    Since $H$ has odd index
and every nontrivial  self dual irreducible Brauer character has even degree,
it follows that the trivial character occurs an odd number of times in $((1_H)^0)^G$.
Arguing as in the proof of Theorem \ref{cor} shows
that the result follows unless every odd order real element has an odd number
of fixed points on $\{Hx \, |\, x \in G\}$.
   Since $[G:H]$ is odd, this   implies the same for any element
whose odd part is real (because if $g \in G$, then   $\langle g_2\rangle$ acts
on the set of  points fixed by $g_{2'}$). Hence $H$ must intersect every conjugacy class
of real elements.   Now by Theorem C,  $K$ is a direct product of copies
of a simple nonabelian group $L$ and $L$ is
isomorphic to $M_{22}$, $M_{23}$ or $\PSL_n(q)$
with $n$ odd.

Since we are excluding the possibility that $L=M_{23}$ or $\PSL_n(q)$ with $q$ even, 
it follows by Theorem C that  $L \cong M_{22}$ or $\PSL_n(q)$
with $n$ odd and $q$ odd and moreover the structure of $M \cap K$
is given in Theorem C.   Arguing as in the case of ordinary characters,
the result follows once we know that $((1_{M \cap L})^0)^L$ has a nontrivial
real constituent.  
If $L=M_{22}$, then the result follows by inspection.  If $G=\PSL_n(q)$
with $q$ odd, then the permutation module in characteristic $2$ 
is $k \oplus W$ with $W$ irreducible (and so self dual).   This follows
from the fact that the smallest representation of $\PSL_n(q)$
with $nq$ odd in characteristic $2$ has dimension ${q^n-1 \over q-1} - 1$
\cite{GPPS}.  
 
The result follows in this case.     Now consider the general case.   Set $J=HK$.
Then we know that $((1_H)^0)^J$ has a nontrivial real constituent $\theta$.   By 
Lemma \ref{below2Br} , then there exists a nontrivial real irreducible Brauer
character $\phi$ such that $[\theta^G, \phi]=1$ whence $\phi$ is a real nontrivial
constituent of $((1_H)^0)^G$.  
\end{proof}

If we strengthen the hypotheses a bit, we can prove the analog of Theorem B
for $2$-Brauer characters
for subgroups of odd index.   We first state an elementary lemma.

\begin{lem} \label{coprimemult}   Let $G$ be a finite group, $H$ a subgroup 
of $G$, and
$A$ a normal subgroup of $G$
of order not divisible by $p$.    Let $\phi$ be an irreducible
$p$-Brauer character with $A \le \ker{\phi}$.  
Then $[((1_H)^0)^G, \phi]=[((1_{HA})^0)^G, \phi]$.   
\end{lem} 

\begin{proof}  Let $k$ be an algebraically closed field of characteristic $p$ and let $M=k[G/H]$
be the permutation module.   Since $|A|$ has order prime to $p$,   $M=[A,M] \oplus C_M(A)$.
Also, note that there is a surjection $f$ of $kG$-modules from $k[G/H]$ onto $k[G/J]$.
Since $A$ acts trivially on $k[G/J]$ and since the dimension of the fixed points of 
$A$ on $k[G/H]$ is $[G:J]$, we see that if $A \le \ker\phi$,  then
$\phi$ is not a composition factor of the kernel of $f$ whence the claim.
\end{proof}

\begin{thm}\label{mainbrauers2}
Let $G$ be a finite group with a proper subgroup $H$ of odd index.  Set $K=\Oh{2'}G$
and  assume that $G =HK$.  
Assume that  $K$ has no composition factor isomorphic to $M_{22}, M_{23}$ or $\PSL_n(q)$ with
$n$ odd. 
Then  the $2$-Brauer character $((1_H)^0)^G$ has a nontrivial real-valued
irreducible constituent $\phi$ of odd multiplicity. 
\end{thm}

\begin{proof}  We induct on $|G|$ and then on $[G:H]$.   
Consider a minimal counterexample.   So  $H$ is corefree and   
$\oh 2 G=1$.   We use induction
on the index.   If there exists $x \in G$ real of odd order with $(1_H)^G(x)$ even,
then as we have seen before the result follows.    So we may assume that $(1_H)^G(x)$ is odd for such
$x$, whence the same is true for all real $x$.   In particular, we may assume that
$H \cap x^G$ is nonempty for all real $x$.  

Suppose there exists a minimal normal abelian subgroup
  $A$ of $G$ contained in $K$.   
  
First suppose that $J:=HA \ne G$.   By minimality, it follow that there exists
a nontrivial real constituent $\phi$ of $((1_J)^0)^G$ of odd multiplicity. 
By Lemma \ref{coprimemult}, $[((1_H)^0)^G, \phi]=[((1_{J})^0)^G, \phi]$ is odd,
a contradiction.

Suppose that $G=J$.   Then $H \cap A$ is normal in $G$, whence $H \cap A=1$.
If a Sylow $2$-subgroup centralizes $A$, then $[K,A]=1$, whence $A \le \zent G$
and so $G= H \times A$ and $K \le H$, a contradiction.   If not, then $A$ contains
real elements but if $1 \ne x \in A$ is real, then $x^G \subseteq A$ is disjoint from
$H$, a contradiction.  

So we may assume that  ${\bf F}(K)=1$.   Let $A$ be a minimal normal subgroup of 
$G$ contained in $K$.   So $A=L_1 \times \ldots \times L_t$ where $L_i \cong L$
is a nonabelian simple group (and not isomorphic to one of the forbidden composition
factors).  Since every real element of $A$ is conjugate to
an element of $H$,  Corollary \ref{cor:oddreal} implies the result.

\end{proof}

We close by noting that if $[G:H]$ is even, it is easy to see that the trivial 2-Brauer
character occurs with even multiplicity in $((1_H)^0)^G$ and so in particular there
are at least two real valued irreducible constituents.

\end{document}